\documentclass[11pt]{article}
\usepackage{amsthm,amsmath,amssymb,amscd,enumerate,graphicx}
\usepackage[all]{xy}
\usepackage[top=40truemm, bottom=40truemm, left=30truemm, right=30truemm]{geometry}

\theoremstyle{plain}
\newtheorem*{Thm}{Main Theorem}
\newtheorem*{Cor}{Corollary}
\newtheorem{thm}{Theorem}

\newtheorem{lem}{Lemma}

\theoremstyle{definition}

\newtheorem{exm}{Example}
\newtheorem{rem}{Remark}
\newtheorem*{Proof}{Proof of the Main Theorem}
\renewcommand{\ker}{\operatorname{Ker}}

\newcommand{\id}{\operatorname{\mathbf{id}}}
\newcommand{\bZ}{\operatorname{\mathbb{Z}}}
\newcommand{\bC}{\operatorname{\mathbb{C}}}

\newcommand{\bN}{\operatorname{\mathbb{N}}}
\newcommand{\bR}{\operatorname{\mathbb{R}}}

\title{$K$-theory for the crossed products of infinite tensor product of $C^*$-algebras by the shift}
\author{Issei Ohhashi}

\begin{document}
\maketitle

\begin{abstract}
C. Schochet shows K\"unneth theorem for the $C^*$-algebras in the smallest class of nuclear $C^*$-algebras which contains the separable Type I algebras and is closed under some operations.
We calculate the $K$-theory for the crossed product of the infinite tensor product of a unital $C^*$-algebra in this class by the shift. 
In particular, we calculate the $K$-theory of the lamplighter group $C^*$-algebra.
\end{abstract}

\section{Introduction}

Let $A$ be a unital $C^*$-algebra
and $A^{\otimes n}$ denote the (minimal) tensor product of $n$ copies of $A$ for each $n \geq 1$.
An \textit{infinite tensor product} $A^{\otimes \bZ}$ of $A$ is the inductive limit of the sequence $\{A^{\otimes 2n-1}\}_n$ with the embeddings $a \mapsto 1_A \otimes a \otimes 1_A$. 
Note that $A^{\otimes \bZ}$ is represented as the closed linear span of
\[\left\{\textstyle\bigotimes \limits_{i\in\bZ} a_i \ :\ a_i \in A, \ 
a_j = 1_A \text{ for all but finitely many } j \right\}.\]
A \textit{shift} $s$ on $A^{\otimes \bZ}$ is the automorphism of $A^{\otimes \bZ}$ satisfying
\begin{align*}
s\left(\textstyle\bigotimes \limits_{i\in\bZ} a_i\right) &= \textstyle\bigotimes \limits_{i\in\bZ} a_{i-1} & \text{ for all } \textstyle\bigotimes \limits_{i\in\bZ} a_i \in A^{\otimes \bZ}.
\end{align*}
We consider the crossed product $A^{\otimes \bZ} \rtimes_s \bZ$ of $A^{\otimes \bZ}$ by $\bZ$ under the action by $s$.

	\begin{exm}\label{ex1}
	\begin{enumerate}
	\item
	Let $C(X)$ denote the $C^*$-algebra of all continuous functions on a compact Hausdorff space $X$ and $\theta$ be an automorphism on $X^{\bZ}$ defined by $(x_i)_{i \in \bZ} \mapsto (x_{i+1})_{i \in \bZ}$.
	Then $C(X^{\bZ}) \rtimes_{\theta} \bZ \cong C(X)^{\otimes \bZ} \rtimes_{s} \bZ$.
	\item
	Let $\Gamma \wr \bZ$ denote the \textit{wreath product} of an amenable discrete group $\Gamma$ with $\bZ$.
	Then, since the group $C^*$-algebra $C^*(\Gamma)$ of $\Gamma$ is nuclear, 
	$C^*(\Gamma \wr \bZ)  \cong C^*(\Gamma)^{\otimes \bZ} \rtimes_s \bZ$.
	\end{enumerate}
	\end{exm}

For a $C^*$-algebra $B$, $K_*(B)$ is the ($\bZ_2$-)graded $K$-theory $K_0(B) \oplus K_1(B)$, where $\bZ_2 := \bZ/2\bZ$.
Let $\mathcal{N}$ denote the smallest class of separable nuclear $C^*$-algebras which contains the separable Type I algebras and is closed under the operations of taking ideals, quotients, extensions, inductive limits, stable isomorphism, and crossed products by $\bZ$ or $\bR$. 
The class $\mathcal{N}$ plays an essential role in K\"unneth theorem by C. Schochet: 

	\begin{thm}[the Spatial Case of K\"unneth Theorem {\cite[Theorem 2.14]{S}}]
	Let $A$ and $B$ be $C^*$-algebras such that $A$ is contained in $\mathcal{N}$. If $K_*(B)$ is torsion free, then there exists a natural isomorphism
	\[\alpha : K_*(A) \otimes K_*(B) \longrightarrow K_*(A \otimes B),\]
	that is, there exist a natural isomorphisms
	\begin{align*}
	(K_0(A) \otimes K_0(B)) \oplus (K_1(A) \otimes K_1(B)) \longrightarrow K_0(A \otimes B),\\
	(K_0(A) \otimes K_1(B)) \oplus (K_1(A) \otimes K_0(B)) \longrightarrow K_1(A \otimes B).
	\end{align*}
	\end{thm}

We introduce the concept of infinite tensor product of graded abelian group.
Let $G = G_0 \oplus G_1$ be an graded abelian group and $G^{\otimes n}$ denote the tensor product of $n$ copies of $G$ (as graded $\bZ$-module). 
For an elements $e$ of $G_0$, an \textit{infinite tensor product} $(G, e)^{\otimes \bZ}$ of $G$ with respect to $e$ is the inductive limit of the sequence $\{G^{\otimes 2n-1}\}_n$ with the embeddings $x \mapsto e \otimes x \otimes e$. 
We often denote $(G, e)^{\otimes \bZ}$ by $G^{\otimes \bZ}$.
Note that $(G, e)^{\otimes \bZ}$ is represented as the linear span of
\[\left\{\textstyle\bigotimes \limits_{i\in\bZ} x_i \ :\ x_i \in G, \ 
x_j = e \text{ for all but finitely many } j \right\}\]
with the grading $(G^{\otimes \bZ})_0 \oplus (G^{\otimes \bZ})_1$, where $(G^{\otimes \bZ})_p$ is the linear span of
\[\left\{\textstyle\bigotimes \limits_{i\in\bZ} x_i \in (G, e)^{\otimes \bZ}\ :\ x_i \in G_{p_i},\  p_i \in \bZ_2,\ \textstyle\sum\limits_{i \in \bZ} p_i =p\right\}.\]
A \textit{shift} $\lambda$ on $(G, e)^{\otimes\bZ}$ is the automorphism of $(G, e)^{\otimes \bZ}$ satisfying
\begin{align*}
\lambda\left(\textstyle\bigotimes \limits_{i\in\bZ} x_i\right) &= \textstyle\bigotimes \limits_{i\in\bZ} x_{i-1}  & \text{for all }
\textstyle\bigotimes \limits_{i\in\bZ} x_i \in (G, e)^{\otimes \bZ}.
\end{align*}

	\begin{rem}\label{rem2}
	Let $A$ be a unital $C^*$-algebra in $\mathcal{N}$, and assume that $K_*(A)$ is a graded \textit{free} abelian group.
	\begin{enumerate}
	\item K\"unneth theorem gives that there exists a natural isomorphism $\alpha_n$ from $K_*(A)^{\otimes n}$ to $K_*(A^{\otimes n})$ for each $n \geq 1$. Form the definition of $\alpha$ \cite[p.446]{S},  
	there exists a natural isomorphism $\alpha_{\infty}$ from $(K_*(A), {[1_A]}_0)^{\otimes \bZ}$ to $K_*(A^{\otimes \bZ})$ such that the following diagram is commutative:
	\begin{align}
	\SelectTips{cm}{}\xymatrix@-0.5pc{
	K_*(A) \ar[r] \ar[d]^{\alpha_1} & K_*(A)^{\otimes 3} \ar[r] \ar[d]^{\alpha_3} & K_*(A)^{\otimes 5} \ar[r] \ar[d]^{\alpha_5} & \cdots \ar[r] &K_*(A)^{\otimes \bZ} \ar@{.>}[d]^{\alpha_{\infty}}\\
	K_*(A) \ar[r] & K_*(A^{\otimes 3}) \ar[r] & K_*(A^{\otimes 5}) \ar[r] & \cdots \ar[r] &K_*(A^{\otimes \bZ}).}
	\end{align}
	\item The automorphism $\lambda := \alpha_{\infty}^{-1} \circ s_* \circ \alpha_{\infty}$ is the shift on $(K_*(A), {[1_A]}_0)^{\otimes \bZ}$.
%


\item Applying $A^{\otimes \bZ} \rtimes_s \bZ$ to Pimsner--Voiculescu theorem \cite[Theorem 10.2.1]{B}, we obtain the following exact triangle:
\begin{align*}
\SelectTips{cm}{}\xymatrixcolsep{0pc}\xymatrix @-0.8pc{
K_*(A^{\otimes \bZ}) \ar[rr]^{\id - s_*}& &
K_*(A^{\otimes \bZ}) \ar[ld]^<<<{ \iota_*}\\
&K_*(A^{\otimes \bZ} \rtimes_s \bZ) \ar[lu]^>>>{\partial},&}
\end{align*}
where $\id - s_*$ and $\iota_*$ have degree $0$ and $\partial$ has degree $1$.
Therefore, by (i) and (ii), we obtain the following exact triangle:
\begin{align}\label{d1}
\SelectTips{cm}{}\xymatrixcolsep{0pc}\xymatrix @-0.8pc{
K_*(A)^{\otimes \bZ} \ar[rr]^{\id - \lambda}& &
K_*(A)^{\otimes \bZ} \ar[ld]^<<<{\iota_* \circ \alpha_{\infty}}\\
&K_*(A^{\otimes \bZ} \rtimes_s \bZ) \ar[lu]^>>>{\alpha_{\infty}^{-1} \circ \partial}.&}
\end{align}
\end{enumerate}
\end{rem}

Let $G$ be a free abelian group which has a basis $E$ containing $e$.
Let $\check{G}$ be the subgroup of $G$ such that
\begin{align}\label{e3}
G = \bZ e \oplus \check{G}.
\end{align}
Let $\check{G} \otimes G^{\otimes \bN}$ denote the linear span of
\[\left\{\textstyle\bigotimes \limits_{i\in\bZ} x_i \in G^{\otimes \bZ}\ \colon\  x_0 \in \check{G},\  x_i = e\text{ for all }i < 0\right\}.\]
Set $H(G, e) := \bZ e^{\otimes \bZ} + \check{G} \otimes G^{\otimes \bN}$ where $e^{\otimes \bZ} := \bigotimes_{i \in \bZ} e$.	
Then, we prove in Lemma \ref{lem3} that 
\[G^{\otimes \bZ} = (\id-\lambda)(G^{\otimes \bZ}) \oplus H(G, e),\]
where $\lambda$ is the shift on $G^{\otimes \bZ}$,
In particular, $H(G, e)$ is naturally isomorphic to the cokernel of $\id-\lambda$.

	\begin{rem}
	\begin{enumerate}
	\item If $G$ is an \textit{ordered} free abelian group with the positive cone $G_{+}$ and $e \in G_{+}$, then $(G, e)^{\otimes \bZ}$ is also an ordered abelian group with the positive cone $G_{+}^{\otimes \bZ}$, all sums of the elements of
	\[\left\{\textstyle\bigotimes \limits_{i\in\bZ} x_i \in (G, e)^{\otimes \bZ} \ :\ x_i \in G_{+} \right\}.\]
	Set $H(G,e)_{+} := H(G,e) \cap G_{+}^{\otimes \bZ}$.
	\item Let $E^{\otimes \bZ}$ denote the basis $\left\{\textstyle\bigotimes_{i \in \bZ} e^{(i)} \ :\ e^{(i)} \in E, \ e^{(j)} = e \text{ for all but finitely many $j$} \right\}$ of $(G, e)^{\otimes \bZ}$.
	Then, there exists a natural isomorphism
	\[\Phi(G, e, E)\ :\ (G, e)^{\otimes \bZ} \longrightarrow \bZ^{\oplus (E^{\otimes \bZ})}.\]
	But, $\Phi(G, e, E)$ may \emph{not} be order-preserving, where the positive cone of $\bZ^{\oplus I}$ is $\bZ_{+}^{\oplus I} := \{(k_m)_{m\in I} \in \bZ^{\oplus I}\ :\ k_m \geq 0\}$ for a set $I$.
	\end{enumerate}
	\end{rem}

	\begin{exm}
	 \begin{enumerate}
	 \item $(\bZ, 1)^{\otimes \bZ} = H(\bZ, 1) = \bZ 1^{\otimes \bZ}$.
	 \item For a natural number $d \geq 2$, let $E_d$ be the standard basis $\{e^{(k)}\}_{k=1}^d$ of $\bZ^d$. 
	 Then, $\Phi(\bZ^d, e^{(1)}, E_d)$ is order-preserving.
	 On the other hand,
	let $\tilde E_d$ be the basis $\{e^{(0)}\} \cap \{e^{(k)}\}_{k=2}^d$ where $e^{(0)} := (1,1, \cdots, 1) \in \bZ^d$.
	 Then, $\Phi(\bZ^d, e^{(0)}, \tilde E_d)$ is \emph{not} order-preserving.
	\end{enumerate}
	\end{exm}

The following is our main theorem:

	\begin{Thm}
	Let $A$ be a $C^*$-algebra in $\mathcal{N}$ and assume that $K_*(A)$ is a free abelian group which has a basis containing ${[1_A]}_0$. 
	Then there exists the following split exact sequence\emph{:}
	\begin{align}\label{d2}
	\SelectTips{cm}{}\xymatrix@-0.5pc{
	0 \ar[r] & H(K_*(A), {[1_A]}_0) \ar[r]^-{\varphi} & K_*(A^{\otimes \bZ} \rtimes_s \bZ) \ar[r]^-{\psi}& \bZ{[1_A]}_0^{\otimes \bZ} \ar[r] &0,}
	\end{align}
	where  $\varphi$ has degree $0$ and $\psi$ has degree $1$. 
	In particular, $\varphi$ is the restriction of $\iota_* \circ \alpha_{\infty}$ and $\psi$ is the homomorphism defined by $x \mapsto (\alpha_{\infty}^{-1} \circ \partial)(x)$. 	
	\end{Thm}
	
	\begin{Cor}
	Let $A$ be a $C^*$-algebra satisfying the assumption of the above theorem.
	If $K_1(A) = 0$, then 
	\begin{align*}
	\SelectTips{cm}{}\xymatrix@-0.5pc{H(K_0(A), {[1_A]}_0) \ar[r]^-{\varphi} & K_0(A^{\otimes \bZ} \rtimes_s \bZ),}&&
	\SelectTips{cm}{}\xymatrix@-0.5pc{K_1(A^{\otimes \bZ} \rtimes_s \bZ) \ar[r]^-{\psi}& \bZ{[1_A]}_0^{\otimes \bZ}}
	\end{align*}
	are isomorphisms. 
	Moreover, if $(K_0(A), K_0(A)_{+})$ is ordered abelian group, then $\varphi$ is order-preserving.
	\end{Cor}


	\begin{exm}
	We calculate the $K$-theory for the group $C^*$-algebra of the {\it lamplighter group}:
	\[\bZ_2\wr \bZ := \bZ_2^{\oplus \bZ} \rtimes \bZ.\]
	Note that 
	\begin{enumerate}
	\item $C^*(\bZ_2\wr \bZ) \cong (\bC^2)^{\otimes \bZ} \rtimes_s \bZ$ \ (Example \ref{ex1}),
	\item $\bC^2 \in \mathcal{N}$,
	\item $(K_0(\bC^2), K_0(\bC^2)_{+}, {[1_{\bC^2}]}_0) \cong (\bZ^2, \bZ_{+}^2, (1,1))$ and $K_1(\bC^2) = 0$.
	\end{enumerate}
	Applying the Main Theorem in the case of $A = \bC^2$, we obtain the following isomorphisms: 
	\begin{align*}
	(K_0(C^*(\bZ_2\wr\bZ)), K_0(C^*(\bZ_2\wr\bZ))_{+})  &\cong (H(\bZ^2, (1,1)), H(\bZ^2, (1,1))_{+}), & K_1(C^*(\bZ_2\wr\bZ))  &\cong  \bZ.
	\end{align*}
%
	\end{exm}

\section{Proof}

In this section, let $G$ be a graded free abelian group which has a basis $E$ containing $e$. 
Set $e^{\otimes n} := \bigotimes_{j=1}^n e \in G^{\otimes n}$ and $G^{\otimes 0} := \bZ$.
Notice that, by (\ref{e3}),
\begin{equation}\label{equation1}
G^{\otimes n} = (e \otimes G^{\otimes n-1}) \oplus (\check{G} \otimes G^{\otimes n-1})
\end{equation}
for each $n \geq 1$.

	\begin{lem}\label{lem1}
	For $n \geq 1$, $y \in G^{\otimes n}$, $k\in \bZ$, and $x \in \check{G} \otimes G^{\otimes n}$, if the equality 
	\begin{equation}\label{equation2}
	y \otimes e - e \otimes y = ke^{\otimes n+1} + x
	\end{equation}
	holds, then $y \in \bZ e^{\otimes n}$, $k=0$, and $x=0$.
	\end{lem}
	
	\begin{proof}
	Note that, by (\ref{equation1}), $y = e \otimes y_1 + x_1$ for some $y_1 \in G^{\otimes n-1}$ and $x_1 \in \check{G} \otimes G^{\otimes n-1}$.
	Then, from (\ref{equation2}), we get the following:
	\[e \otimes (y_1\otimes e - y - k e^{\otimes n}) + (x_1 \otimes e - x) = 0.\]
	Hence, by (\ref{equation1}), $y_1\otimes e - y - k e^{\otimes n}=0$ and $x_1 \otimes e - x = 0$, so
	\[y = y_1 \otimes e - k e^{\otimes n}, \hspace{5mm} x = x_1 \otimes e.\]
	Therefore, from (\ref{equation2}), we get $(y_1 \otimes e - e \otimes y_1) \otimes e = (k e^{\otimes n} + x_1) \otimes e$,
	so
	\[y_1 \otimes e - e \otimes y_1 = k e^{\otimes n} +  x_1.\]
	Note that, replacing $(n-1, y_1, x_1)$ of this equation with $(n, y, x)$, we get (\ref{equation2}).   
	
	Repeating this argument, we can find $y_j \in G^{\otimes n-j}$ for each integer $j$ $(1 \leq j \leq n)$ such that
	\[y_{j-1} = y_j \otimes e - k e^{\otimes n-j+1},\]
	where $y_0:=y$.
	Inductively, we obtain $y = (y_n -n k) e^{\otimes n} \in \bZ e^{\otimes n}$. 
	Hence, from (\ref{equation2}), $e \otimes (k e^{\otimes n}) + x = 0$. 
	Thus, by (\ref{equation1}),  $k=0$ and $x=0$.
	\end{proof}



	For each $n \geq 1$, identify $G^{\otimes 2n-1}$ and the linear span of
	\[\left\{\textstyle\bigotimes \limits_{i\in\bZ} x_i \in G^{\otimes \bZ}\ \colon\  x_i = e \text{ for all }|i| > 2n-1 \right\}.\]	

	\begin{lem}\label{lem2}
	Let $\lambda$ be the shift on $G^{\otimes \bZ}$. For $z \in G^{\otimes \bZ}$ and $w \in H(G, e)$, if the equality 
	\begin{equation}\label{e2}
	(\id-\lambda)(z) = w
	\end{equation}
	holds, then $z \in \bZ e^{\otimes \bZ}$ and $w=0$.
	\end{lem}

	\begin{proof}
	Take $n \geq 1$,  satisfying $(\id-\lambda)(z)$ $(=w) \in G^{\otimes 2n+1}$.
	Note that 
	\begin{align*}
	(\id - \lambda)(G^{\otimes \bZ}) \cap G^{\otimes 2n+1} &=\{e \otimes z' - z' \otimes e \ :\ z' \in G^{\otimes 2n}\},\\
	H(G, e) \cap G^{\otimes 2n+1} &=\bZ e^{\otimes 2n+1} + e^{\otimes n} \otimes \check{G} \otimes G^{\otimes n}.
	\end{align*}
	Hence, $(\id-\lambda)(z) = e \otimes z' - z' \otimes e$ and $w = k e^{\otimes 2n+1} + e^{\otimes n} \otimes x$ for some $z' \in G^{\otimes 2n}$, $k \in \bZ$, and $x \in \check{G} \otimes G^{\otimes n}$.
	Therefore, by (\ref{e2}),
	\[z' \otimes e - e \otimes z' = k e^{\otimes n+1} + e^{\otimes n} \otimes x .\]
	Inductively, this equality follows that $z' = e^{\otimes n} \otimes y$ for some $y \in G^{\otimes n}$.
	Hence
	\[y \otimes e - e \otimes y = k e^{\otimes n+1} + x.\]
	By Lemma \ref{lem1}, we obtain $y \in \bZ e^{\otimes n}$, $k=0$, and $x=0$.
	Thus, $z \in \bZ e^{\otimes \bZ}$ and $w=0$.
	\end{proof}

	\begin{lem}\label{lem3}
	Let $\lambda$ be the shift on $G^{\otimes \bZ}$.
	Then, 
	\begin{enumerate}[{\rm(i)}]
	\item $\ker(\id - \lambda) = \bZ e^{\otimes \bZ}$,\label{item1}
	\item $G^{\otimes \bZ} = (\id -\lambda) (G^{\otimes \bZ}) \oplus H(G, e)$.\label{item3}
	\end{enumerate}
	\end{lem}
	
	\begin{proof}
	(\ref{item1}) At first, it is clear that $\ker(\id - \lambda) \supset \bZ e^{\otimes \bZ}$.
	Conversely, take $z \in G^{\otimes \bZ}$ such that $(\id - \lambda)(z) =0$.
	Applying Lemma \ref{lem2} to the case of $w=0$, we obtain that $z \in \bZ e^{\otimes \bZ}$. Thus $\ker(\id - \lambda) = \bZ e^{\otimes \bZ}$.
	
	(\ref{item3}) 
	By Lemma \ref{lem1}, $(\id -\lambda) (G^{\otimes \bZ}) \cap H(G, e) = 0$,
	so it is sufficient to show $E^{\otimes \bZ} \subset (\id -\lambda) (G^{\otimes \bZ}) \oplus H(G, e)$.
	Take $f = \bigotimes_{i\in\bZ} e_i \in E^{\otimes \bZ}$.
	Since $e^{\otimes \bZ} \in H(G, e)$, we may assume $f \neq e^{\otimes \bZ}$. 
	Then, there exists an integer $i_0\in\bZ$ such that $e_{i_0} \neq e$ and $e_i = e$ for all $i < i_0$.
	If $i_0=0$, then $f \in  H(G, e)$, so we may assume $i_0\neq 0$.
	Note that $\lambda^{-i_0}(f) \in \check{G} \otimes G^{\otimes\bN}$ and
	\[f = \left\{\begin{array}{ll}
	(\id -\lambda) \left(-\textstyle\sum\limits_{j=1}^{i_0} \lambda^{-j}(f)\right) + \lambda^{-i_0}(f) & \text{for }i_0 > 0\\
	(\id -\lambda) \left(\textstyle\sum\limits_{j=1}^{-i_0} \lambda^{-j+1}(f)\right) + \lambda^{-i_0}(f) & \text{for }i_0 < 0,\\
	\end{array}\right.\]
	Thus, $f \in (\id -\lambda) (G^{\otimes \bZ}) \oplus H(G, e)$.
	\end{proof}

	\begin{Proof}
	Applying Lemma \ref{lem3} in the case of $G=K_*(A)$ and $e = {[1_A]}_0$, we obtain
	\begin{align*}
	\ker(\id - \lambda) &  = \bZ{[1_A]}_0^{\otimes \bZ}, \\
	K_*(A)^{\otimes \bZ} &  = (\id -\lambda)(K_*(A)^{\otimes \bZ}) \oplus H(K_*(A), {[1_A]}_0).
	\end{align*}
	Therefore, from the exactness of (\ref{d1}), the sequence (\ref{d2}) is exact.  
	Moreover, since $\bZ{[1_A]}_0^{\otimes \bZ} \cong \bZ$, (\ref{d2}) is split.
	\qed
	\end{Proof}

\vspace{5mm}

\noindent
\textit{Acknowlegement}:
I would like to thank Tsuyoshi Kato for introducing to this subject and for encouraging discusses.
%
%

\end{document}